\documentclass[11pt]{article}
\usepackage{amsmath,amsfonts,amssymb,amsthm,color}
\begin{document}
\newcommand{\1}{{{\bf 1}}}
\newcommand{\id}{{\rm id}}
\newcommand{\Hom}{{\rm Hom}\,}
\newcommand{\End}{{\rm End}\,}
\newcommand{\Image}{{\rm Im}\,}
\newcommand{\Ind}{{\rm Ind}\,}
\newcommand{\Aut}{{\rm Aut}\,}
\newcommand{\Ker}{{\rm Ker}\,}
\newcommand{\gr}{{\rm gr}}
\newcommand{\Der}{{\rm Der}\,}
\newcommand{\Res}{{\rm Res}}

\newcommand{\Z}{\mathbb{Z}}
\newcommand{\Q}{\mathbb{Q}}
\newcommand{\C}{\mathbb{C}}
\newcommand{\N}{\mathbb{N}}
\newcommand{\g}{\mathfrak{g}}
\newcommand{\gl}{\mathfrak{gl}}
\newcommand{\h}{\mathfrak{h}}
\newcommand{\wt}{{\rm wt}\;}
\newcommand{\A}{\mathcal{A}}
\newcommand{\D}{\mathcal{D}}
\newcommand{\Lie}{\mathcal{L}}

\def \<{\langle}
\def \>{\rangle}
\def \be{\begin{equation}\label}
\def \ee{\end{equation}}
\def \bex{\begin{exa}\label}
\def \eex{\end{exa}}
\def \bl{\begin{lem}\label}
\def \el{\end{lem}}
\def \bt{\begin{thm}\label}
\def \et{\end{thm}}
\def \bp{\begin{prop}\label}
\def \ep{\end{prop}}
\def \br{\begin{rem}\label}
\def \er{\end{rem}}
\def \bc{\begin{coro}\label}
\def \ec{\end{coro}}
\def \bd{\begin{de}\label}
\def \ed{\end{de}}

\newtheorem{thm}{Theorem}[section]
\newtheorem{prop}[thm]{Proposition}
\newtheorem{coro}[thm]{Corollary}
\newtheorem{conj}[thm]{Conjecture}
\newtheorem{exa}[thm]{Example}
\newtheorem{lem}[thm]{Lemma}
\newtheorem{rem}[thm]{Remark}
\newtheorem{de}[thm]{Definition}
\newtheorem{hy}[thm]{Hypothesis}
\makeatletter \@addtoreset{equation}{section}
\def\theequation{\thesection.\arabic{equation}}
\makeatother \makeatletter

\begin{Large}
\begin{center}
\textbf{Semi-conformal structure on certain vertex superalgebras 
associated to vertex superalgebroids}
\end{center}
\end{Large}

\begin{center}
{ Ming Li$^{a}$\footnote{Partially supported by Scholarship Program of Xiamen University (Nos.Y03108)}\\
$\mbox{}^{a}$School of Mathematical Sciences, Xiamen University,
Xiamen 361005, China}\\
Email: mingli@stu.xmu.edu.cn
\end{center}

\begin{abstract}
In this paper, we first  give the definiton of a vertex superalgebroid. Then we construct  a family of vertex superalgebras associated to vertex superalgebroids. 
As a main result, we find a sufficient and necessary condition that this vertex superalgebras are semi-conformal. In addition, we give an concrete example of this vertex superalgebras and apply our results to this superalgebra.
\end{abstract}


\section{Introduction}
In the general theory of vertex algebras, an important difference between vertex algebras and vertex operator algebras is the existence of an internal Virasoro algebra module structure on for vertex operator algebras. But in fact, for some research on vertex operator algebra  we use only the $sl_{2}={\rm span}\{L_0, L_{\pm1} \}$-module action. Thus in \cite{fhl}, a notion called quasi-vertex operator algebra arises. Namely a subalgebra of a given vertex operator algebra with assuming only the existence and properties of  $sl_{2}$-action.  A semi-conformal structure on a vertex superalgebra is a module structure for the subalgebra of the Virasoro algebra, generated by $L_m$ with $m\ge -1$, satisfying certain conditions (Ref.\cite{FB}). This structure is essential to the fundamental theory of contragredient modules and symmetric bilinear forms, and it is the minimal structure necessary in the study on modular invariance of characters of modules. 

The notion of a vertex algebroid, first introduced in \cite{gms}, captures a family of  $\N$-graded vertex algebra. From the vertex algebra point of view, a vertex algebroid is a vector space with 3 partially defined operations satisfying a number of disparate identities that make some sense only if one discerns the Borcherds identity lurking behind (Ref.\cite{Ma}).
In \cite{li-y}, Li and Yamskulna classify all the  $\N$-graded simple modules of those vertex algebras in terms of simple modules for certain Lie algebroids. In \cite{li-ya}, they construct and classify graded simple twisted modules for those vertex algebras. In addition, they determine the full automorphism groups of those  vertex algebras in terms of the automorphism groups of the corresponding vertex algebroids. 

In the present paper, firstly we give the definiton of a vertex superalgebroid and the associated vertex superalgebra. The main purpose in this paper is to determine semi-conformal structure on those vertex superalgebras associated to vertex superalgebroids.
By definition, a semi-conformal vertex superalgebra is a vertex superalgebra  equipped with a semi-conformal structure. The main result is  a sufficient and necessary condition that the associated vertex superalgebras are semi-conformal. In addition, we give an concrete example of such vertex superalgebra and apply our results to this superalgebra.

This paper is organized as follows. In Section 2, we review some basic notations, formulas and properties for Lie superalgebras and vertex superalgebras. We give  the definitions  of $1$-truncated conformal superalgebra, vertex superalgebroid and Lie superalgebroid, which are generalization of these algebras in non-super version. We also review tools from \cite{gms} and \cite{li-y}, then we construct  an $\N$-graded  vertex superalgebra for a given vertex superalgebroid. In Section 3, we recall the definition of semi-conformal vertex superalgebra and find a sufficient and necessary condition that the given vertex superalgebras  are semi-conformal. In Section 4, we give an example of this family of vertex superalgebras.
\section{Preliminaries}
 
 We use the usual symbols ${\bf \Z}$, ${\bf\Z}_+$ and $\N$ for the set of integers,
 the positive integers, and   the nonnegative integers respectively. 

Let $M=M_{\bar{0}}\oplus M_{\bar{1}}$ be any superalgebra, i.e., $({\bf \Z}/2{\bf \Z})$-graded algebra. Any element $u$ in $M_{\bar{0}}$ (resp. $M_{\bar{1}}$) is said to be
{\it even} (resp. {\it odd}).
For any homogeneous element $u$, we define $|u|=0$ if $u$
is even, $|u|=1$ if $u$ is odd. We define
$\varepsilon_{u,v}=(-1)^{|u||v|}$, for any homogeneous elements $u,v\in
M$. We note that, the space of endomorphisms of $M$, denoted by ${\rm End} (M)$ is a superalgebra.

Throughout this paper, when we write $|u|$ for an element $ u\in M$, we will always implicitly assume that $u$ is a homogeneous element.
\begin{de}
A Lie superalgebra is a superalgebra $A=A_{\bar{0}}\oplus A_{\bar{1}}$ with multiplication $[\cdot, \cdot]$ satisfying the following two axioms: for homogeneous $a, b, c\in A,$
\begin{eqnarray*}
&&skew-supersymmetry: [a, b]=-\varepsilon_{a,b}[b, a].\\
&&super\,\, Jacobi\,\, identity: [a, [b, c]]=[[a, b], c]+\varepsilon_{a,b}[b, [a, c]].
\end{eqnarray*}
\end{de}

\begin{exa} 
Let $A=A_{\bar{0}}\oplus A_{\bar{1}}$ be an associative superalgebra. A becomes a Lie superalgebra with
\begin{eqnarray}\label{LIE}
[a, b]=ab-\varepsilon_{a,b}ba, \,\,a, b\in A.
\end{eqnarray}
We say that $A$ is super-commutative if $[a, b]=0$ for all $a, b\in A.$
\end{exa}

Let $A$ be a Lie superalgebra. Then ${\rm End}(A)$ is an associative superalgebra, and hence it carries a structure of Lie superalgebra by  (\ref{LIE}).

\begin{de}
Let $A=A_{\bar{0}}\oplus A_{\bar{1}}$ and $A'=A'_{\bar{0}}\oplus A'_{\bar{1}}$ be superalgebras. A linear map $\varphi: A\rightarrow A'$ is said to be even if $\varphi(A_\theta)\subseteq A'_\theta, \forall\,\theta \in {\bf \Z}/2{\bf \Z}$.  
\end{de}

\begin{de}

Let $A=A_{\bar{0}}\oplus A_{\bar{1}}$ be an associative superalgebra. An endomorphism $D\in {\rm End}(A)_{s}$ is called a derivation of degree $s$, if it satisfies the identity
\begin{eqnarray}
D(ab)=D(a)b+(-1)^{s|a|}aD(b), \,\,a, b\in A.
\end{eqnarray}
Denote by ${\rm Der}(A)_s$ the space of derivations on $A$ of degree $s$. 
\end{de}

Now we review some basic  notions of a vertex superalgebra (Refs.\cite{B86},\cite{kac2} and \cite{ll}).
\begin{de}A  vertex superalgebra is a triple
$(V,{\bf 1},Y)$, where $V=V_{\bar{0}}\oplus V_{\bar{1}}$ is a superspace, ${\bf 1}\in V_{\bar{0}}$ is a specified vector called the {\it vacuum} of $V$,
and $Y$ is a linear map
\begin{eqnarray*}
Y(\cdot,z):& &V\rightarrow ({\rm End}V)[[z,z^{-1}]];\nonumber\\
& &a\mapsto Y(a,z)=\sum_{n\in{\bf \Z}}a_{n}z^{-n-1}\;\;(\mbox{where }
a_{n}\in {\rm End}V)
\end{eqnarray*}
such that
\begin{eqnarray*}
(V1)& &\mbox{For any }a,b\in V, a_{n}b=0\;\mbox{ for }n
\mbox{ sufficiently large;}\\
(V2)& &\mbox{For  }\bar{i}, \bar{j}\in {\bf \Z}/2{\bf \Z}, a\in V_{\bar{i}}, b\in V_{\bar{j}},  Y(a,z)b\in V_{\bar{i}+\bar{j}}[[z,z^{-1}]];\\
(V3)& &Y({\bf 1},z)={\rm id}_{V}\;\;\;\mbox{(the identity operator of $V$)};\\
(V4)& &Y(a,z){\bf 1}\in V[[z]] \mbox{ and }\lim_{z \rightarrow
0}Y(a,z){\bf 1}=a\;\;\mbox{  for any }a\in V;\\
(V5)& &\mbox{For homogeneous elements }a,b\in V,
\mbox{ the following {\it super Jacobi identity} holds:}
\end{eqnarray*}
\begin{eqnarray}
& &z_{0}^{-1}\delta\left(\frac{z_{1}-z_{2}}{z_{0}}\right)Y(a,z_{1})Y(b,z_{2})
-\varepsilon_{a,b}z_{0}^{-1}\delta\left(\frac{z_{2}-z_{1}}{-z_{0}}\right)
Y(b,z_{2})Y(a,z_{1})\nonumber \\
&=&z_{2}^{-1}\delta\left(\frac{z_{1}-z_{0}}{z_{2}}\right)Y(Y(a,z_{0})b,z_{2}).
\end{eqnarray}
\end{de}

From the super Jacobi identity we have Borcherds' super  commutator formula and super iterate formula in component form:
\begin{eqnarray}
[u_{m},v_{n}]&=&\sum_{i\ge 0}\binom{m}{i}(u_{i}v)_{m+n-i},\label{bor}\\
(u_{m}v)_{n}&=&\sum_{i\ge 0}(-1)^{i}\binom{m}{i}
\left( u_{m-i}v_{n+i}-\varepsilon_{u,v}(-1)^{m}v_{m+n-i}u_{i}\right)
\end{eqnarray}
for $u,v\in V,\; m,n\in \Z$, where $[u_{m},v_{n}]=u_{m}v_{n}-\varepsilon_{u,v}v_{n}u_{m}$.

Define a linear operator $\D$ on $V$ by $\D(v)=v_{-2}{\bf 1}$, we have following consequences:
\begin{eqnarray}
& &Y(a,z)b=\varepsilon_{a,b}e^{z\D}Y(b,-z)a\;\;\;\mbox{ for any homogeneous}\,\,a, b\in V.\label{D1}\\
& & [\D,a_{n}]=(\D a)_{n}=-n a_{n-1}\;\;\;\mbox{ for any } a\in V,\; n\in \Z. \label{D2}
\end{eqnarray}

\begin{de}

Let $(V,{\bf 1},Y)$ be a vertex superalgebra. A
$V$-module is a triple $(M, d, Y_{M})$, where $M$ is a superspace, $d$ is an even endomorphism of $M$, and $Y_{M}$ is
a linear map

\begin{eqnarray*}
Y_{M}(\cdot,z):& &V\rightarrow ({\rm End}M)[[z,z^{-1}]];\nonumber\\
& &a\mapsto Y_{M}(a,z)=\sum_{n\in{\bf \Z}}a_{n}z^{-n-1}\;\;(\mbox{where }
a_{n}\in {\rm End}M)
\end{eqnarray*}
 satisfying the
following conditions:
\begin{eqnarray*}
(M1)& &\mbox{For any }a\in V,u\in M, a_{n}u=0\;\mbox{ for }n
\mbox{ sufficient large};\\
(M2)& &Y_{M}({\bf 1},z)={\rm id}_{M};\\
(M3)& &[d,Y_{M}(a,z)]=Y_{M}(D(a),z)={d\over dz}Y_{M}(a,z)\;\;\mbox{
for any }a\in V;\\
(M4)& &\mbox{For homogeneous }a,b\in V,\mbox{ the
following super Jacobi identity holds:}
\end{eqnarray*}
\begin{eqnarray}
& &z_{0}^{-1}\delta\left(\frac{z_{1}-z_{2}}{z_{0}}\right)Y_{M}(a,z_{1})
Y_{M}(b,z_{2})-\varepsilon_{a,b}z_{0}^{-1}\delta\left(\frac{z_{2}-z_{1}}
{-z_{0}}\right)Y_{M}(b,z_{2})Y_{M}(a,z_{1})\nonumber\\
&=&z_{2}^{-1}\delta\left(\frac{z_{1}-z_{0}}{z_{2}}\right)
Y_{M}(Y(a,z_{0})b,z_{2}).
\end{eqnarray}
\end{de}

\begin{de}
A vertex superalgebra $V$ equipped with a $\Z$-grading
 $V=\coprod_{n\in
\Z}V_{n}=V_{\bar{0}}\oplus V_{\bar{1}}$=$\coprod_{n\in\Z}(V_{\bar{0},n}\oplus V_{\bar{1},n})$ is called a {\em $\Z$-graded vertex superalgebra} if ${\bf
1}\in V_{\bar{0},0}$ and  for $u\in V_{\bar{i}, k}$ with $k\in \Z$, $m,n\in \Z$ and $\bar{i}, \bar{j}\in {\bf \Z}/2{\bf \Z}$,
\begin{eqnarray}
u_{m}V_{\bar{j}, n}\subset V_{\bar{i}+\bar{j},\, k+n-m-1}.
\end{eqnarray}
\end{de}
Now we give the definition of 1-truncated conformal superalgebra (Refs.\cite{gms} and \cite{kac2}).
\begin{de}
{\em A {\em $1$-truncated conformal
superalgebra} $C$ is a superspace $C=C_{0}\oplus C_{1}$=$\bigoplus_{ i=0,1}C_{\bar{0},i}\oplus C_{\bar{1},i}$  equipped
with an even linear map $\partial: C_{0}\rightarrow C_{1}$ and a bilinear
product $u_{i}v$ for $i=0,1$ such that the
following axioms hold:

(Derivation) for $a\in C_{0},\; u\in C_{1}$,
\begin{eqnarray}\label{e2.7}
(\partial a)_{0}=0,\;\;(\partial a)_{1}=-a_{0},\;\;
\partial (u_{0}a)=u_{0}\partial a;
\end{eqnarray}

(Super commutativity) for $a\in C_{0},\; u,v\in C_{1}$,
\begin{eqnarray}\label{e2.8}
u_{0}a=-\varepsilon_{u, a}a_{0}u,\;\; u_{0}v=\varepsilon_{u, v}(-v_{0}u+\partial (v_{1}u)),\;\;
u_{1}v=\varepsilon_{u, v}v_{1}u;
\end{eqnarray}

(Super associativity) for $\alpha,\beta,\gamma\in C$,
\begin{eqnarray}\label{e2.9}
\alpha_{0}(\beta_{i}\gamma)
=\varepsilon_{\alpha,\beta}\beta_{i}(\alpha_{0}\gamma) +(\alpha_{0}\beta)_{i}\gamma.
\end{eqnarray}}
\end{de}

Next we  recall the notion of Leibniz superalgebra from \cite{D1}.
\begin{de}
A superspace $\Gamma=\Gamma_{\bar{0}}\bigoplus \Gamma_{\bar{1}}$ with multiplication $[\cdot, \cdot]$ is
called a Leibniz superalgebra, if it satisfies the following conditions:
\begin{eqnarray*}
&&[\Gamma_{\alpha}, \Gamma_{\beta}]\subseteq \Gamma_{\alpha+\beta} \,\,\,for \
all \ \alpha,\beta\in {\bf
\Z}/2{\bf \Z},\\
&&[u, [v, w]]=[[u, v], w]+\varepsilon_{u,v}[v, [u, w]],
\end{eqnarray*}
 for all $u, v, w\in \Gamma$. The second condition is called graded  Leibniz   identity.
\end{de}

Now we introduce the definition of a vertex superalgebroid (Refs.\cite{gms} and \cite{C}).

\bd{dalgebroid}
{\em Let $A$ be a (unital) super-commutative associative superalgebra (over $\C$).
A {\em vertex $A$-superalgebroid} is a $\C$-vector space $\Gamma$ equipped with

(0) a $\C$-bilinear map
\begin{eqnarray}
A\times \Gamma \rightarrow \Gamma;\;\;\;
(a,v)\mapsto a*v
\end{eqnarray}
such that $1*v=v$ (i.e., a ``non-associative unital $A$-module'')

(1) a structure of a Leibniz superalgebra
$[\cdot,\cdot]:
    \Gamma\otimes_{\C}\Gamma\rightarrow \Gamma$

(2) a homomorphism of Leibniz superalgebras
     $\pi: \Gamma\rightarrow {\rm Der}(A)$

(3) a $\C$-bilinear pairing
$\<\cdot,\cdot\>: \Gamma\otimes_{\C}\Gamma\rightarrow A$, such that $\<u,v\>=\varepsilon_{u, v}\<v,u\>.$

(4) a $\C$-linear map $\partial: A\rightarrow \Gamma$ such that
    $\pi\circ \partial =0$,\\
which satisfy the following conditions:
\begin{eqnarray}
a*(a'*v)-(aa')*v&=&\varepsilon_{a, a'}\pi(v)(a)* \partial (a')+\pi(v)(a')*\partial (a)\qquad \quad \label{pe2.22}\\
{[u,a*v]}&=&\pi(u)(a)*v+\varepsilon_{u, a}a*[u,v]\label{pe2.23}\\
{[u,v]+\varepsilon_{u, v}[v,u]}&=& \partial(\<u,v\>)\label{e2.23}\\
\pi (a*v)&=& a\pi(v)\label{epi-hom}\\
\<a*u,v\>&=& a\<u,v\>-\varepsilon_{a, u}\varepsilon_{a, v}\pi(u)(\pi(v)(a))\label{e2.26}\\
\pi(v)(\<v_{1},v_{2}\>)&=&\<[v,v_{1}],v_{2}\>+\varepsilon_{v, v_{1}}\<v_{1},[v,v_{2}]\>
\label{esymmetricmap}\\
\partial (aa')&=& a* \partial (a')+\varepsilon_{a, a'}a'* \partial (a)\label{epartial-der}\\
{[v,\partial (a)]}&=& \partial (\pi(v)(a))\label{epartial-hom}\\
\<v,\partial (a)\>&=& \pi(v)(a)\label{e2.29}
\end{eqnarray}
for $a,a'\in A,\; u, v,v_{1}, v_{2}\in \Gamma$.}
\ed
The following proposition could be proved by the similar discussion as in Proposition 2.11 in \cite{li-y}.
\begin{prop}\label{pact} 
Let $A$ be a unital  super-commutative associative superalgebra
and $B$ a module for $A$ as a nonassociative superalgebra.
Then a vertex $A$-superalgebroid structure on $B$ exactly amounts to
a $1$-truncated conformal superalgebra structure on $C=A\oplus B$ with
\begin{eqnarray}
& &a_{i}a'=0,\\
& &u_{0}v=[u,v],\;\; \ \ u_{1}v=\<u,v\>,\\
& &u_{0}a=\pi(u)(a),\;\; \ \varepsilon_{u, a}a_{0}u=-u_{0}a=-\pi(u)(a)
\end{eqnarray}
for $a,a'\in A,\; u,v\in B,\; i=0,1$, such that
\begin{eqnarray}
& &a(a'u)-(aa')u=\varepsilon_{a, a'}(u_{0}a)\partial a'+(u_{0}a')\partial a,\label{eabu-abu}\\
& &u_{0}(av)-\varepsilon_{u, a}a(u_{0}v)=(u_{0}a)v,\label{eu0av2.36}\\
& &u_{0}(aa')=\varepsilon_{u, a}a(u_{0}a')+(u_{0}a)a',\label{u0ab2.35}\\
& &a'_{0}(av)=\varepsilon_{a, a'}a(a'_{0}v),\label{ea0a'v}\\
& &(au)_{1}v=a(u_{1}v)-\varepsilon_{a, u}\varepsilon_{a, v}u_{0}v_{0}a,\label{eau1v}\\
& &\partial (aa')=a\partial a'+\varepsilon_{a, a'}a'\partial a.\label{epab}
\end{eqnarray}
\end{prop}

In what follows, we construct the vertex superalgebra associated to vertex superalgebroid, which is following the description of \cite{li-y}.

Let $C=A\oplus B$ be a $1$-truncated conformal superalgebra. Set
$L(A\oplus B)=(A\oplus B)\otimes \C[t,t^{-1}].$ Then we define a super commutator bilinear product $[\cdot,\cdot]$ on $L(A\oplus B)$ by
\begin{eqnarray}
[a\otimes t^{m}, a'\otimes t^{n}]&=&0,\label{edef-bracket-1}\\
{[a\otimes t^{m}, b\otimes t^{n}]}&=&a_{0}b\otimes t^{m+n},\label{edef-bracket-2}\\
{[b\otimes t^{n},a\otimes t^{m}]}&=&b_{0}a\otimes
t^{m+n},\label{edef-bracket-22}\\
{[b\otimes t^{m}, b'\otimes t^{n}]}&=&b_{0}b'\otimes
t^{m+n}+m(b_{1}b')\otimes t^{m+n-1}\label{edef-bracket-3}
\end{eqnarray}
for $a,a'\in A,\; b,b'\in B,\;m,n\in \Z$.

For $a\in A, b\in B\;n\in \Z$, define $\deg (a\otimes t^{n})=-n-1$, $\deg (b\otimes t^{n})=-n.$
Then $L(A\oplus B)$ becomes a $\Z$-graded superalgebra.
Set
\begin{eqnarray*}
\hat{\partial}=\partial\otimes 1+1\otimes d/dt: 
\;\; L(A)\rightarrow L(A\oplus B)
\end{eqnarray*}
and ${\mathcal{L}}(A\oplus B)=L(A\oplus B)/\hat{\partial}L(A).$

By the similar discussions as Proposition 3.1, Proposition 3.2, Proposition 3.3 in [LY1], we have that  the nonassociative superalgebra $\Lie (A\oplus B)$ is a $\Z$-graded Lie superalgebra. 

For the rest of this section 
we denote this Lie superalgebra by $\Lie$. 
Let $\rho$ be the projection map from $L(A\oplus B)$ to ${\mathcal{L}}$.
For $u\in A\oplus B,\; n\in \Z,$ set
$$u(n)=\rho(u\otimes t^{n})=u\otimes t^{n}+\hat{\partial}L(A)\in
{\mathcal{L}}$$ 
Define
\begin{eqnarray}
\Lie^{\ge 0}&=&\rho ((A\oplus B)\otimes \C[t])\subset \Lie,\\
\Lie^{< 0}&=&\rho ((A\oplus B)\otimes t^{-1}\C[t^{-1}])\subset\Lie.
\end{eqnarray}

Consider $\C$ as the trivial $\Lie^{\ge 0}$-module and then form the
following generalized Verma module
\begin{equation}
V_{\Lie}=U(\Lie)\otimes _{U(\Lie^{\ge 0})}\C,
\end{equation}
where $U(\Lie)$ be the universal enveloping algebra of  $\Lie$.

Recall the Poincar\'{e}-Birkhoff-Witt theorem of Lie superalgebra (Refs.\cite{Ro} and \cite {M}).
\bt{PBW}Let $g=g_{\bar{0}}\oplus g_{\bar{1}}$ be a Lie superalgebra. Let $\{x_1, x_2,\cdots, x_p\}$ be a basis for $g_{\bar{0}}$ and $\{y_1, y_2,\cdots, y_q\}$ be a basis for $g_{\bar{1}}$. The set $$\{x_1^{r_1}x_2^{r_2}\cdots x_p^{r_p}y_1^{s_1}y_2^{s_2}\cdots y_q^{s_q}\mid r_1, r_2,\cdots, r_p\in\N, s_1, s_2,\cdots, s_q\in\{0, 1\}\}$$
is a basis for $U(g)$.
 \et

In view of the Poincar\'{e}-Birkhoff-Witt theorem, we
have
\begin{eqnarray}\label{eva-pbw}
V_{\Lie}=U(\Lie^{<0}),
\end{eqnarray}
as a vector space, so that we may consider $A\oplus
B$ as a subspace:
\begin{eqnarray}\label{eidentificationAB}
A\oplus B\rightarrow V_{\Lie};\;\; a+b\mapsto
a(-1){\bf 1}+b(-1){\bf 1}.
\end{eqnarray}
We set ${\bf 1}=1\in V_{\Lie}$ and $\deg \C=0$. Then
$V_{\Lie}$ is an $\N$-graded $\Lie$-module:
\begin{eqnarray}\label{egrading-valie}
V_{\Lie}=\coprod_{n\in\N}(V_{\Lie})_{(n)},
\end{eqnarray}
and  there exists a unique vertex superalgebra structure on $V_{\Lie}$. ( Refs.\cite{flm}, \cite{li-y} and \cite{ll}.)

\bl{lprepare}
Set
\begin{eqnarray}
E_{0}&=&{\rm span} \{ e-{\bf 1},\; a(-1)a'-aa'\;|\; a,a'\in A\}
\subset (V_{\Lie})_{(0)},\\
E_{1}&=&{\rm span} \{ a(-1)b-ab\;|\; a\in A,\; b\in B\}
\subset (V_{\Lie})_{(1)},\\
E&=&E_{0}\oplus E_{1}\\
&=&{\rm span} \{ e-{\bf 1},\; a(-1)a'-aa',\;a(-1)b-ab\;|\; a,a'\in A,\; b\in
B\}\subset V_{\Lie}.\nonumber
\end{eqnarray}
Then $v(n)E\subset E$, for $v\in C=A\oplus B,\; n\ge 0.$
Furthermore, we have
$\D E_{0}\subset E_{1}$, $B(-1)E_{0}\subset A(-1)E_{1}+E_{1}.$
\el

\begin{proof} 
 Using the similar method of Lemma 4.2 in \cite{li-y}, we can prove the first assertion and $\D E_{0}\subset E_{1}$.
 Finally, we prove the last assertion.
For $b\in B$, we have
\begin{eqnarray*}
b(-1)(e-{\bf 1})=b(-1)e-b=e(-1)b-\D e(0)b-eb=e(-1)b-eb\in E_{1},
\end{eqnarray*}
using the fact that $e(n)=0$ for $n\ne -1$.
Let $b\in B,\; a,a'\in A$. Using Borcherds' super commutator formula (\ref{bor}), the $\D$-bracket
formula (\ref{D1}) and (\ref{D2}), and the fact that $\D E_{0}\subset E_{1}$, we have
\begin{eqnarray*}
& &b(-1)(a(-1)a'-aa')\\
&=&\varepsilon_{a, b}a(-1)b(-1)a'+(b_{0}a)(-2)a'-b(-1)(aa')\\
&=&\varepsilon_{a, b}\varepsilon_{a', b}a(-1)a'(-1)b-\varepsilon_{a, b}\varepsilon_{a', b}a(-1)\D a'(0)b +\D (b_{0}a)(-1)a'-(b_{0}a)(-1)\D a'\\
& &\ \ \ \ -\varepsilon_{a, b}\varepsilon_{a', b}(aa')(-1)b+\varepsilon_{a, b}\varepsilon_{a', b}\D (aa')(0)b\\
&=&\varepsilon_{a, b}\varepsilon_{a', b}a(-1)a'(-1)b-\varepsilon_{a, b}\varepsilon_{a', b}a(-1)\D a'(0)b +\D (b_{0}a)(-1)a'-(b_{0}a)(-1)\D a'\\
& &\ \ \ \ -\varepsilon_{a, b}\varepsilon_{a', b}(aa')(-1)b-\D b_{0}(aa')\\
&=&\varepsilon_{a, b}\varepsilon_{a', b}a(-1)a'(-1)b-\varepsilon_{a, b}\varepsilon_{a', b}a(-1)\D a'(0)b +\D (b_{0}a)(-1)a'-(b_{0}a)(-1)\D a'\\
& &\ \ \ \ -\varepsilon_{a, b}\varepsilon_{a', b}(aa')(-1)b-\D (\varepsilon_{a, b}a(b_{0}a')+(b_{0}a)a')\\
&=&\varepsilon_{a, b}\varepsilon_{a', b}a(-1)a'(-1)b-\varepsilon_{a, b}\varepsilon_{a', b}a(-1)\D a'(0)b +\D (b_{0}a)(-1)a'-(b_{0}a)(-1)\D a'\\
& &\ \ \ \ -\varepsilon_{a, b}\varepsilon_{a', b}(aa')(-1)b-\varepsilon_{a, b}a\D (b_{0}a')-\varepsilon_{a, a'}(b_{0}a')\D a-\D ((b_{0}a)a')\\
&=&\varepsilon_{a, b}\varepsilon_{a', b}a(-1)a'(-1)b -(b_{0}a)(-1)\D a'-\varepsilon_{a, b}\varepsilon_{a', b}(aa')(-1)b-\varepsilon_{a, a'}(b_{0}a')\D a\\
& &+\varepsilon_{a, b}\left(a(-1)\D b_{0}a'-a\D (b_{0}a')\right)
+\D \left((b_{0}a)(-1)a'-(b_{0}a)a'\right)\\
&\equiv& \varepsilon_{a, b}\varepsilon_{a', b}a(-1)a'(-1)b -(b_{0}a)\D a'-\varepsilon_{a, b}\varepsilon_{a', b}(aa')b-\varepsilon_{a, a'}(b_{0}a')\D a\;\;\;\mbox{mod }E_{1}\\
&=&\varepsilon_{a, b}\varepsilon_{a', b}a(-1)a'(-1)b -\varepsilon_{a, b}\varepsilon_{a', b}a(a'b)\\
&=& \varepsilon_{a, b}\varepsilon_{a', b}a(-1)(a'(-1)b-a'b)+\varepsilon_{a, b}\varepsilon_{a', b}\left(a(-1)(a'b)-a(a'b)\right)\\
&\in& A(-1)E_{1}+E_{1}.
\end{eqnarray*}
This proves that $B(-1)E_{0}\subset A(-1)E_{1}+E_{1}$.
\end{proof}

Define $I_{B}=U(\Lie)\C[\D]E.$
Using the proof of Proposition 4.4 in \cite{li-y}, we have that $I_{B}$ is a two-sided graded ideal of $V_{\Lie}$. Define $V_{B}=V_{\Lie}/I_{B}.$

Using similar discussion of \cite{gms} and \cite{li-y}, we have 
\bt{tA-B}
Let $B$ be a vertex $A$-superalgebroid and let $V_{B}$ be the 
associated $\N$-graded vertex superalgebra. 
We have $(V_{B})_{(0)}=A$ and $(V_{B})_{(1)}=B$ (under the linear map
$v\mapsto v(-1){\bf 1}$) and
$V_{B}$ as a vertex superalgebra is generated by $A\oplus B$.
Furthermore, for any $n\ge 1$,
\begin{eqnarray*}
& &(V_{B})_{(n)}\\
&=&{\rm span}\{b_1(-n_1)...b_k(-n_k){\bf 1}\;|\;b_i\in B,\;
n_1\geq n_2\geq ...\geq n_k\geq 1,\; n_1+...+n_k=n\}.
\end{eqnarray*}
\et

\section{ Semi-conformal structure on vertex superalgebra $V_{B}$}
In this section, we study semi-conformal structure on the vertex superalgebra $V_{B}$
 associated to a vertex $A$-superalgebroid $B$.

Recall that the Virasoro algebra $Vir$ is a Lie algebra  with a basis 
$\{ L_{n}\ |\ n\in \Z\}\cup \{ c\}$, where 
\begin{eqnarray}
[L_m,L_n]=(m-n)L_{m+n}+\frac{1}{12}(m^3-m)\delta_{m+n,0}c
\end{eqnarray}
for $m,n\in \Z$, and $c$ is central. 
We denote by $L(n)$ the operator on any $Vir$-module corresponding to $L_n$ for $n\in \Z$.

Set
\begin{eqnarray}
Vir^+={\rm span}\{L_m \ |\  m\ge -1\},
\end{eqnarray}
which  is a subalgebra of the Virasoro algebra $Vir$. The elements $L_{\pm 1}$ and $L_0$ span a 
subalgebra, which is isomorphic to  the  Lie algebra $sl_{2}$, where 
\begin{equation*}
[L_0,L_{\pm 1}]=\mp L_{\pm 1}, \  \   \  \  [L_1,L_{-1}]=2L_0.
\end{equation*}
Let $\mathfrak{b}$ denote the $2$-dimensional subalgebra of $Vir^{+}$, spanned by $L_0$ and $L_1$.
By a {\em weight $\mathfrak{b}$-module} we mean a $\mathfrak{b}$-module on which $L(0)$ acts semisimply.

For convenience, we set $\Z_{\ge -1}=\{ m\in \Z\ |\; m\ge -1\}.$

\begin{de}
{\em A {\em 1-truncated $sl_{2}$-module} is a weight $\mathfrak{b}$-module $U=U_0\oplus U_1$,
where $ L(0)|_{U_0}=0,  \  L(0)|_{U_1}=1$, equipped with a linear map
$L(-1)\in {\rm Hom} (U_0,U_1)$.}
\end{de}

Note that for any  weight $\mathfrak{b}$-module $U=U_0\oplus U_1$,  we have   
\begin{eqnarray*}
  L(1): U_1\longrightarrow U_0, \   \  L(1)|_{U_0}=0.
\end{eqnarray*}

\begin{lem}\label{general}
Let $U$ be a $\mathfrak{b}$-module. Set $L(U)=U\otimes \C[t,t^{-1}]$.
Then  $L(U)$ is a $Vir^+$-module with the action given by
\begin{eqnarray}\label{module-action}
L(m)* (u\otimes t^n)&=&-(m+n+1)(u\otimes t^{m+n})+(m+1)(L(0)u\otimes t^{m+n})\nonumber\\
&&\   +  \frac{1}{2}m(m+1) (L(1)u\otimes t^{m+n-1})
\end{eqnarray}
for $m\in \Z_{\ge -1},\  u\in U,\ n\in \Z$.
\end{lem}

\begin{proof} Let $m,n\ge -1$ and $u\in U,\ r\in \Z$. Then we need to prove 
$$[L(m), L(n)]* (u\otimes t^{r})=L(m)* (L(n)*(u\otimes t^r))-L(n)* (L(m)* (u\otimes t^r)).$$
Indeed, by definition, we have
\begin{eqnarray*}
&&L(m)* (L(n)* (u\otimes t^r))\\
&=&-(n+r+1)L(m)* (u\otimes t^{n+r})+(n+1)L(m)* (L(0)u\otimes t^{n+r})\\
&& +  \frac{1}{2}n(n+1) L(m)*  (L(1)u\otimes t^{n+r-1})\\
&=&(n+r+1)(m+n+r+1) (u\otimes t^{m+n+r})\\
&& -(n+r+1)(m+1) (L(0)u\otimes t^{m+n+r})\\
&&-\frac{1}{2}m(m+1)(n+r+1)(L(1)u\otimes t^{m+n+r-1})\\
&&-(n+1)(m+n+r+1)(L(0)u\otimes t^{m+n+r})\\
&&+(n+1)(m+1)(L(0)^2u\otimes t^{m+n+r})\\
&&+\frac{1}{2}m(m+1)(n+1)(L(1)L(0)u\otimes t^{m+n+r-1})\\
&&- \frac{1}{2}n(n+1)(m+n+r)(L(1)u\otimes t^{m+n+r-1})\\
&&+ \frac{1}{2}n(n+1)(m+1)(L(0)L(1)u\otimes t^{m+n+r-1})\\
&&+ \frac{1}{4}n(n+1)m(m+1)(L(1)L(1)u\otimes t^{m+n+r-2}).
\end{eqnarray*}
On the other hand, by switching $m$ with $n$ we have
\begin{eqnarray*}
&&L(n)* (L(m)* (u\otimes t^r))\\
&=&(m+r+1)(m+n+r+1) (u\otimes t^{m+n+r})\\
&& -(m+r+1)(n+1) (L(0)u\otimes t^{m+n+r})\\
&&-\frac{1}{2}n(n+1)(m+r+1)(L(1)u\otimes t^{m+n+r-1})\\
&&-(m+1)(m+n+r+1)(L(0)u\otimes t^{m+n+r})\\
&&+(m+1)(n+1)(L(0)^2u\otimes t^{m+n+r})\\
&&+\frac{1}{2}n(n+1)(m+1)(L(1)L(0)u\otimes t^{m+n+r-1})\\
&&- \frac{1}{2}m(m+1)(m+n+r)(L(1)u\otimes t^{m+n+r-1})\\
&&+ \frac{1}{2}m(m+1)(n+1)(L(0)L(1)u\otimes t^{m+n+r-1})\\
&&+ \frac{1}{4}n(n+1)m(m+1)(L(1)L(1)u\otimes t^{m+n+r-2}).
\end{eqnarray*}
Then we obtain
\begin{eqnarray*}
&&L(m)* (L(n)* (u\otimes t^r))-L(n)* (L(m)* (u\otimes t^r))\\
&=&(n-m)(m+n+r+1) (u\otimes t^{m+n+r})\\
&&+(m(m+1)-n(n+1))(L(0)u\otimes t^{m+n+r})\\
&&+\frac{1}{2} ( m(m+1)(m-1)-n(n+1)(n-1))(L(1)u\otimes t^{m+n+r-1})\\
&&+\frac{1}{2}(m-n)(m+1)(n+1) (L(1)L(0)u\otimes t^{m+n+r-1})\\
&&+\frac{1}{2}(n-m)(m+1)(n+1) (L(0)L(1)u\otimes t^{m+n+r-1})\\
&=&(n-m)(m+n+r+1) (u\otimes t^{m+n+r})\\
&&+(m-n)(m+n+1)(L(0)u\otimes t^{m+n+r})\\
&&+\frac{1}{2} (m-n)(m^2+mn+n^2-1)(L(1)u\otimes t^{m+n+r-1})\\
&&+\frac{1}{2}(m-n)(m+1)(n+1) ([L(1),L(0)]u\otimes t^{m+n+r-1})\\
&=&(n-m)(m+n+r+1) (u\otimes t^{m+n+r})\\
&&\ \ +(m-n) (m+n+1)(L(0)u\otimes t^{m+n+r})\\
&&\ \ +\frac{1}{2}(m+n)(m+n+1)(m-n) (L(1)u\otimes t^{m+n+r-1})\\
&=&(m-n)L(m+n)* (u\otimes t^{r})\\
&=&[L(m), L(n)]*(u\otimes t^{r}),
\end{eqnarray*}
as desired. 
Therefore, $L(U)$ is a  $Vir^+$-module.
\end{proof}

For the rest of this section, we assume that $B$ is a vertex $A$-superalgebroid and $V_{B}$ is 
the associated $\N$-graded vertex superalgebra with $(V_{B})_{(0)}=A$ and $(V_{B})_{(1)}=B$. 
Recall that $V_{B}$ as a vertex superalgebra is generated by $A\oplus B$. 

\begin{lem}\label{derivation-action-Lie-algebra}
Let $B$ be a vertex $A$-superalgebroid equipped with a weight $\mathfrak{b}$-module structure
 on $A\oplus B$ with $L(0)|_A=0$ and $L(0)|_{B}=1$ such that $L(1)(\partial A)=0$.
Then the $Vir^+$-module structure on $L(A\oplus B)$ obtained in Lemma \ref{general}  
reduces to a $Vir^+$-module structure on ${\mathcal{L}}(A\oplus B):=L(A\oplus B)/\hat{\partial}L(A)$.
Furthermore, if in addition we assume
\begin{eqnarray}
L(1)(u_0v)=(L(1)u)_0v+u_0(L(1)v)\   \   \   \mbox{ for }u,v\in B,
\end{eqnarray}
then $Vir^+$ acts on ${\mathcal{L}}(A\oplus B)$ as a  Lie algebra of derivations.
\end{lem}

\begin{proof}  For the first assertion, we need to show that $ \hat{\partial}L(A)$ is a $Vir^+$-submodule. 
Indeed, for $m\ge -1,\ a \in A,\ n\in \Z$, noticing that $L(0)\partial a=\partial a$, $L(0)a=0=L(1)a$, and 
$L(1){\partial}a=0$ (by assumption), we have
\begin{eqnarray*}
&&L(m)*\hat{\partial}(a\otimes t^{n})\\
&=&L(m)*({\partial}a\otimes t^{n}+na\otimes t^{n-1})\\
&=&-(m+n+1)({\partial}a\otimes t^{m+n})+(m+1)(L(0){\partial}a\otimes t^{m+n})+ \frac{1}{2}(m+1)m(L(1){\partial}a\otimes t^{m+n-1}) \\
&&-n(m+n)(a\otimes t^{m+n-1})+n(m+1)(L(0)a\otimes t^{m+n-1})
+\frac{1}{2}nm(m+1)(L(1)a\otimes t^{m+n-2})  \\
&=&-n({\partial}a\otimes t^{m+n})-n(m+n)a\otimes t^{m+n-1}\\
&=&-n\hat{\partial}(a\otimes t^{m+n}).
\end{eqnarray*}
This proves that $ \hat{\partial}L(A)$ is a $Vir^+$-submodule. 
Then, ${\mathcal{L}}(A+B)$ is naturally a $Vir^+$-module. 

Next, we show that $Vir^+$ acts on ${\mathcal{L}}(A\oplus B)$ as a derivation algebra. 
It suffices to prove
\begin{eqnarray}\label{derivation-proof}
L(s)*[u\otimes t^{m}, v\otimes t^{n}]=[L(s)*(u\otimes t^{m}), v\otimes t^{n}]+[u\otimes t^{m}, L(s)*(v\otimes t^{n})]
\end{eqnarray}
for $m, n, s\in \Z,\  u, v\in A\cup B$ with $s\ge -1$.

Note that for $a\in A$, as $L(0)a=0=L(1)a$, we have
$$L(p)*(a\otimes t^q)=-(p+q+1)(a\otimes t^{p+q})\   \  \   \mbox{ for }p,q\in \Z\  \mbox{ with }p\ge -1.$$
On the other hand, we have $[a\otimes t^p,a'\otimes t^q]=0$ for any $a,a'\in A,\ p,q\in \Z$.
We see that (\ref{derivation-proof}) holds for $u,v\in A$.

Consider the case with $u\in A, \ v\in B$. 
For $m, n, s\in \Z$ with $s\ge -1$, noticing that $u_0v\in A$ and $L(0)(u_0v)=0=L(1)(u_0v)$, we have
\begin{eqnarray*}
L(s)*[u\otimes t^{m}, v\otimes t^{n}]&=&L(s)*(u_0v\otimes t^{m+n})
=-(s+m+n+1)(u_0v)\otimes t^{m+n+s},\\
\ [L(s)*(u\otimes t^{m}), v\otimes t^{n}]&=&-(s+m+1)[u\otimes t^{m+s}, v\otimes t^{n}]\\
&=&-(s+m+1)(u_0v)\otimes t^{m+n+s},\\
\ [u\otimes t^{m}, L(s)*(v\otimes t^{n})]&=& [u\otimes t^{m}, -(s+n+1)(v\otimes t^{n+s})+(s+1) (L(0)v\otimes t^{n+s})\\
&&+\frac{1}{2}s(s+1) (L(1)v\otimes t^{n+s})]\\
&=&-(s+n+1)[u\otimes t^{m}, v\otimes t^{n+s}]+(s+1)[u\otimes t^{m}, v\otimes t^{n+s}]\\
&=&-n(u_0v)\otimes t^{m+n+s}.
\end{eqnarray*}
Then it implies that (\ref{derivation-proof}) holds.

For the case with $u\in B,\ v\in A$, the result follows from the skew-supersymmetry of Lie superalgebra.
 
Consider the case with $u, v\in B$.  We have 
\begin{eqnarray*}
&&[u\otimes t^{m}, L(s)*(v\otimes t^{n})]\\
&=&\Big[u\otimes t^{m}, -(n+s+1)(v\otimes t^{s+n})+(s+1)(L(0)v\otimes t^{s+n})\\
&&+\frac{1}{2}s(s+1) (L(1)v\otimes t^{s+n-1})\Big]\\
&=&-n\left[u\otimes t^{m}, v\otimes t^{s+n}\right]+\frac{1}{2}s(s+1)\left[u\otimes t^{m}, L(1)v\otimes t^{s+n-1}\right]
\\
&=&-n\Big(u_0v\otimes t^{m+n+s}+m(u_1v)\otimes t^{m+n+s-1}\Big)\\
&&+\frac{1}{2}s(s+1)u_0L(1)v\otimes t^{s+m+n-1},\\
\\
&&[L(s)*(u\otimes t^{m}), v\otimes t^{n}]\\
&=&[-(s+m+1)(u\otimes t^{s+m})+(s+1)(L(0)u\otimes t^{s+m})\\
&&+\frac{1}{2}s(s+1)(L(1)u\otimes t^{s+m-1}), v\otimes t^{n}]
\\
&=&[-(s+m+1)(u\otimes t^{s+m})+(s+1)(u\otimes t^{s+m})\\
&&+\frac{1}{2}s(s+1)(L(1)u\otimes t^{s+m-1}), v\otimes t^{n}]
\\
&=&-m\left(u_0v\otimes t^{m+n+s}+(s+m)(u_1v)\otimes t^{m+n+s-1}\right)\\
&&+\frac{1}{2}s(s+1)(L(1)u)_0v\otimes t^{s+m+n-1},\\
\\
&&L(s)*[u\otimes t^{m}, v\otimes t^{n}]\\
&=&L(s)*\left(u_0v\otimes t^{m+n}+mu_1v\otimes t^{m+n-1}\right)\\
&=&-(m+n+s+1)(u_0v)\otimes t^{m+n+s}+(s+1)L(0)(u_0v)\otimes t^{m+n+s}\\
&&+\frac{1}{2}s(s+1)L(1)(u_0v)\otimes t^{m+n+s-1}-m(m+n+s)u_1v\otimes t^{m+n+s-1}\\
&&+m(s+1)L(0)u_1v\otimes t^{m+n+s-1}
+\frac{1}{2}s(s+1)mL(1)u_1v\otimes t^{m+n+s-2}.
\end{eqnarray*}
As $L(1)(u_0v)=(L(1)u)_0v+u_0(L(1)v)$ by assumption, it follows that (\ref{derivation-proof}) holds.
\end{proof}

\begin{de}\label{def-semi-conformal}
{\em A {\em semi-conformal vertex superalgebra} is a vertex superalgebra $V$ equipped with 
a $Vir^{+}$-module structure satisfying the conditions that 
$V=\oplus_{n\in \Z}V_{(n)}$,  where $V_{(n)}=\{ v\in V\ |\  L(0)v=nv\}$ for $n\in \Z$, 
 $V_{(-n)}=0$ for $n$ sufficiently large, $L(-1)=\D$ on $V$, and that 
\begin{eqnarray}
[L(m), Y(v,x)]=\sum_{i\ge 0}\binom{m+1}{i} x^{m+1-i} Y(L(i-1)v,x)
\end{eqnarray}
for $m\in \Z_{\ge -1},\ v\in V$.}
\end{de}

Let $B$ be a vertex $A$-superalgebroid  as in Lemma \ref{derivation-action-Lie-algebra}
with all the assumptions.
Recall that for $u\in A+B,\ n\in \Z$,  $u(n)$ denotes the image of 
$u\otimes t^n$ in Lie superalgebra ${\mathcal{L}}(A+ B)=L(A+B)/\hat{\partial}L(A)$.
We have
\begin{align}\label{L(m)-on-u(n)}
&L(m)* u(n)\nonumber\\
=&-(m+n+1)u(m+n)+(m+1)(L(0)u)(m+n)+\binom{m+1}{2}(L(1)u)(m+n-1)
\end{align}
for $u\in A+B,\ m,n\in \Z$ with $m\ge -1$. 
We denote this Lie superalgebra by $\Lie$.
Recall that the vertex superalgebra $V_{\Lie}$ as an $\Lie$-module is the induced module
$$V_{\Lie}=U(\Lie)\otimes_{U(\Lie^{\ge 0})}\C,$$
where $\C$ is viewed as a trivial $\Lie^{\ge 0}$-module. 
The space $A\oplus B$ is identified with a subspace of $V_{\Lie}$ through the linear map
$u\mapsto u(-1){\bf 1}$.
For $u\in A+B$, we have
$$Y(u,x)=u(x)=\sum_{n\in \Z}u(n)x^{-n-1}.$$
Thus $u_{n}=u(n)$ on $V_{\Lie}$.

\begin{lem}\label{Vir-on-V-L}
Let $B$ be a vertex $A$-superalgebroid with a weight $\mathfrak{b}$-module structure
on $A\oplus B$ with $L(0)|_A=0$, $L(0)|_B=1$  such that $L(1)\partial A=0$ and
\begin{eqnarray*}
L(1)(u_0v)=(L(1)u)_0v+u_0(L(1)v)\   \   \   \mbox{ for }u,v\in B.
\end{eqnarray*}
 Then there exists a $Vir^+$-module structure on $V_{\Lie}$, which is uniquely determined by the condition
 that  $Vir^+\cdot {\bf 1}=0$ and
 \begin{align}
& [L(m),u(n)]\nonumber\\
=&-(m+n+1)u(m+n)+(m+1)(L(0)u)(m+n)+\binom{m+1}{2}(L(1)u)(m+n-1)
 \end{align}
 for $u\in A+B,\ m\in \Z_{\ge -1},\ n\in \Z$. Furthermore, the map $u\mapsto u(-1){\bf 1}$ is a $\mathfrak{b}$-module homomorphism with $L(-1)=\D$,   $L(-1)a=\partial a$ for $a\in A$, and $L(m)(A+B)=0$ for $m\ge 2$.
\end{lem}

\begin{proof} Since $V_{\Lie}$ is an $\Lie$-module generated by ${\bf 1}$,
the uniqueness is clear. It remains to verify the existence.
Since $Vir^{+}$ acts on $\Lie$ as a Lie algebra of derivations, $Vir^{+}$ acts on
 the universal enveloping algebra $U(\Lie)$ of $\Lie$ as a Lie algebra of derivations. 
 Especially, we have $Vir^{+}\cdot 1=0$.
Note that $V_{\Lie}$ as a $U(\Lie)$-module is isomorphic to $U(\Lie)/J_B$, where $J_B=U(\Lie)\Lie^{\ge 0}$. 
From (\ref{L(m)-on-u(n)}) we see that the action of $Vir^{+}$ on $\Lie$ preserves $\Lie^{\ge 0}$.  
It follows that $J_B$ is a $Vir^{+}$-submodule.
 Consequently, $V_{\Lie}$ is a $Vir^{+}$-module such that $Vir^{+}\cdot {\bf 1}=0$ and
 \begin{eqnarray}
 &&L(m) u(n)w-u(n)L(m)w\nonumber\\
 &=&(L(m)*u(n))w\nonumber\\
 &=&-(m+n+1)u(m+n)w+(m+1)(L(0)u)(m+n)w\nonumber\\
 &&+\frac{1}{2}m(m+1)(L(1)u)(m+n-1)w
 \end{eqnarray}
for $u\in A+B,\ w\in V_{\Lie},\ m,n\in \Z$ with $m\ge -1$. More specifically, we have
  \begin{eqnarray}
  &&[L(m),a(n)]=-(m+n+1)a(m+n),\\
  &&[L(m),b(n)]= -nb(m+n)+\frac{1}{2}m(m+1)(L(1)b)(m+n-1)
  \end{eqnarray}
 on $V_{\Lie}$  for $a\in A,\ b\in B,\ m,n\in \Z$ with $m\ge -1$.  In particular, we have
 \begin{eqnarray}
 &&[L(-1),u(n)]=-nu(n-1),\\
 &&[L(0),a(n)]=-(n+1)a(n),\   \  \  \   [L(0),b(n)]=-n b(n)
 \end{eqnarray}
 for $u\in A+B,\  a\in A,\ b\in B,\ n\in \Z$. 
 
 Recall that $\D {\bf 1}=0$ and $[\D, u(n)]=-nu(n-1)$. For $a\in A$, we have
\begin{eqnarray*}
\D a=a_{-2}{\bf 1}=a(-2){\bf 1}=(\partial a)_{-1}{\bf 1}=\partial a.
\end{eqnarray*}
Since  $V_{\Lie}$ as an $\Lie$-module is generated by ${\bf 1}$ and $L(-1){\bf 1}=0$, it follows that 
 $L(-1)=\D$ on $V_{\Lie}$.  Similarly, as $L(0){\bf 1}=0$,  it follows that $L(0)$ is semisimple on $V_{\Lie}$
 with only integer eigenvalues. In particular,  we have
 \begin{eqnarray*}
&& L(0)a(-1){\bf 1}=a(-1)L(0){\bf 1}=0,\ L(0)b(-1){\bf 1}=  b(-1)L(0){\bf 1}+b(-1){\bf 1}=b(-1){\bf 1},\\
&& L(1)a(-1){\bf 1}=a(-1)L(1){\bf 1}-a(0){\bf 1}=0,\\ 
&&L(1)b(-1){\bf 1}=  b(-1)L(1){\bf 1}+b(0){\bf 1}+(L(1)b)(-1){\bf 1}=(L(1)b)(-1){\bf 1}.
\end{eqnarray*}
This proves that the map $u\mapsto u(-1){\bf 1}$ is a $\mathfrak{b}$-module homomorphism.
Furthermore, using the commutation relations of $L(m)$ with $a(n)$ and $b(n)$, we get $L(m)(A+B)=0$ for $m\ge 2$.
This completes the proof.
\end{proof}


As the main result of this section,  we have: 

\begin{thm}
Let $B$ be a vertex $A$-superalgebroid with a weight $\mathfrak{b}$-module structure
on $A\oplus B$ with $L(0)|_A=0$, $L(0)|_B=1$  such that $L(1)\partial A=0$ and
\begin{eqnarray}
L(1)(u_0v)=(L(1)u)_0v+u_0(L(1)v)\   \   \   \mbox{ for }u,v\in B.\label{de}\\
 L(1)(ab)=aL(1)b-a_0b \  \  \  \mbox{ for  }a\in A,\  b\in B.\label{xyzw}
 \end{eqnarray}
Then the $\mathfrak{b}$-module structure on $A+B$ can be extended to a $Vir^+$-module structure on $V_{B}$, 
which is uniquely determined by
\begin{eqnarray}
&&L(-1)=\D, \\  
&&[L(m), Y(u,x)]=\sum_{i=0}^{2}\binom{m+1}{i} x^{m+1-i} Y(L(i-1)u,x)\label{L(m)-Yv-bracket}
\end{eqnarray} 
for $u\in A+B, \  m\in \Z_{\ge -1}$.
Furthermore, $V_{B}$ is a semi-conformal vertex superalgebra.
\end{thm}

\begin{proof}  Just as $V_{\Lie}$ in Lemma \ref{Vir-on-V-L}, the uniqueness is clear.
From Lemma \ref{Vir-on-V-L}, there is a $Vir^+$-module structure on $V_{\Lie}$.  
Recall that $V_{B}=V_{\Lie}/I_{B}$, where
$I_{B}$ is the ideal of $V_{\Lie}$ generated by the subset
$$E=\{ e-{\bf 1},\;a(-1)a'-aa',\;a(-1)b-ab \;|\; a,a'\in A,\; b\in B\}.$$
Now, we prove that $I_B$ is a $Vir^+$-submodule, so that $V_{B}$ is naturally a $Vir^+$-module.
Note that $I_B$ is the $\Lie$-submodule generated by $\D^k E$ for $k\ge 0$.
Since $L(-1)=\D$ on $V_{\Lie}$,  it suffices to show that $L(m)E\subset E$ for  $m\ge 0$.

For $m\ge 0$,  we have
\begin{eqnarray*}
L(m)(e-{\bf 1})=L(m)e(-1){\bf 1}-L(m){\bf 1}=-me(m-1){\bf 1}+e(-1)L(m){\bf 1}-L(m){\bf 1}=0
\end{eqnarray*}
and 
\begin{eqnarray*}
L(m)(a(-1)a'-aa')=-ma(m-1)a'+a(-1)L(m)a'-L(m)(aa')=0
\end{eqnarray*}
for $a,a'\in A$ as $a'(n)a=0$ for all $n\ge 0$.

For $a\in A,\ b\in B$, we have 
\begin{eqnarray*}
L(0)(a(-1)b-ab)=a(-1)L(0)b-L(0)(ab)
=a(-1)b-ab.
\end{eqnarray*}
Using (\ref{xyzw}), we also have
\begin{eqnarray*}
L(1)(a(-1)b-ab)=-a(0)b+a(-1)(L(1)b)-L(1)(ab)
=a(-1)(L(1)b)-a(L(1)b),
\end{eqnarray*}
which lies in $E$ with $L(1)b\in A$. Furthermore, for $m\ge 2$, we have
$$L(m)(a(-1)b-ab)=-ma(m-1)b+a(-1)L(m)b-L(m)(ab)=0.$$
In summary, we have $L(m)E\subset E$ for $m\ge 0$.

To prove that $V_{B}$ is a semi-conformal vertex superalgebra, it remains to prove
\begin{eqnarray*}
[L(m), Y(v,x)]=\sum_{i\ge 0}\binom{m+1}{i} x^{m-i+1} Y(L(i-1)v,x)
\end{eqnarray*}
for $m\ge -1,\ v\in V_B$. Let $V$ consist of vectors $v\in V_B$ such that the above equality holds.
We now prove $V=V_{B}$. As $V_{B}$ is a vertex superalgebra generated by $A+ B$,
it suffices to prove that $V$ is a vertex subalgebra which 
contains $A+B$. 
Let $u\in A+B$. 
For $m\ge -1$,  we have
 \begin{eqnarray*}
 &&[L(m), Y(u,x)] \  \  (=[L(m),u(x)])\\
 &=&-\sum_{n\in \bf \Z}(m+n+1)u_{m+n}x^{-n-1}+\sum_{n\in \bf \Z}(m+1)(L(0)u)_{m+n}x^{-n-1}\\
&&\ \ +\binom{m+1}{2}\sum_{n\in \bf \Z}(L(1)u)_{m+n-1}x^{-n-1}\\
&=&x^{m+1}Y(L(-1)u,x)+(m+1)x^{m}Y(L(0)u,x)
+\binom{m+1}{2}x^{m-1}Y(L(1)u,x)\\
&=&\sum_{i\ge 0}\binom{m+1}{i} x^{m-i+1} Y(L(i-1)u,x)
\end{eqnarray*} 
as $L(n)u=0$ for $n\ge 2$.
Thus $A+B\subset V$.
On the other hand, since $Y({\bf 1},x)=1$ and $L(i-1){\bf 1}=0$ for $i\ge0$, we see that ${\bf 1}\in V$.

To prove that $V$ is a vertex subalgebra, we must show that $V$ is closed, which amounts to that
for any $u,v\in V,\ m\in \Z_{\ge -1}$,
\begin{align}\label{L(m)-YYuv}
[L(m), Y(Y(u,x_0)v, x)]=\sum_{i\ge 0}\binom{m+1}{i} x^{m-i+1} Y(L(i-1)Y(u,x_0)v, x).
\end{align}
Now, let $u, v\in V$ and let $m\in \Z_{\ge -1}$.  Using Jacobi identity, we have
\begin{eqnarray*}
& &[L(m), Y(Y(u,x_0)v, x)]\\
&=&\Res_{x_1}[L(m), x_{0}^{-1}\delta\left(\frac{x_{1}-x}{x_{0}}\right)Y(u,x_{1})Y(v,x)
-\varepsilon_{u,v}x_{0}^{-1}\delta\left(\frac{x-x_{1}}{-x_{0}}\right)
Y(v,x)Y(u,x_{1})] \\
&=&\Res_{x_1}x_{0}^{-1}\delta\left(\frac{x_{1}-x}{x_{0}}\right)
\big(\sum_{i\ge 0}\binom{m+1}{i} x_1^{m+1-i} Y(L(i-1)u,x_1)Y(v,x)\\
& &\   \  \   \  +\sum_{i\ge 0}\binom{m+1}{i} x^{m+1-i}Y(u,x_1)Y(L(i-1)v,x)\big)\\
& & -\varepsilon_{u, v}\Res_{x_1}x_{0}^{-1}\delta\left(\frac{x-x_1}{-x_{0}}\right)
\big(\sum_{i\ge 0}\binom{m+1}{i} x^{m+1-i} Y(L(i-1)v,x)Y(u,x_1)\\
& &\   \   \   \   +\sum_{i\ge 0}\binom{m+1}{i} x_1^{m+1-i}Y(v,x)Y(L(i-1)u,x_1)\big)\\
&=&\Res_{x_1}\sum_{i\ge 0}\binom{m+1}{i}x_{1}^{-1}\delta\left(\frac{x+x_0}{x_{1}}\right) x_1^{m+1-i}Y(Y(L(i-1)u,x_0)v,x)\big)\\
& &+\sum_{i\ge 0}\binom{m+1}{i} x^{m+1-i}Y(Y(u,x_0)L(i-1)v,x)\\
&=&\sum_{i\ge 0}\binom{m+1}{i}(x+x_0)^{m+1-i}Y(Y(L(i-1)u,x_0)v,x)\\
& &+\sum_{i\ge 0}\binom{m+1}{i} x^{m+1-i}Y(Y(u,x_0)L(i-1)v,x).
\end{eqnarray*}
On the other hand, we have
\begin{eqnarray*}
&&\sum_{i\ge 0}\binom{m+1}{i}x^{m+1-i}Y([L(i-1),Y(u,x_0)]v,x)\\
&=&\sum_{i,r\ge 0}\binom{m+1}{i}\binom{i}{r}x^{m+1-i}x_{0}^{i-r}Y(Y(L(r-1)u,x_0)v,x)\\
&=&\sum_{r,j\ge 0}\binom{m+1}{r+j}\binom{r+j}{r}x^{m+1-r-j}x_{0}^{j}Y(Y(L(r-1)u,x_0)v,x)\\
&=&\sum_{r,j\ge 0}\binom{m+1}{r}\binom{m+1-r}{j}x^{m+1-r-j}x_{0}^{j}Y(Y(L(r-1)u,x_0)v,x)\\
&=&\sum_{r\ge 0}\binom{m+1}{r}(x+x_0)^{m+1-r}Y(Y(L(r-1)u,x_0)v,x),
\end{eqnarray*}
where we use the fact that $\binom{m+1}{r+j}\binom{r+j}{r}=\binom{m+1}{r}\binom{m+1-r}{j}$.
Combining the last two equations we obtain (\ref{L(m)-YYuv}). 
Therefore, $V$ is a vertex subalgebra that contains $A+B$, and hence
we have $V=V_{B}$, which proves that (\ref{L(m)-Yv-bracket}) holds for all $m\in \Z_{\ge -1},\  v\in V_B$.
As we have proved  in Lemma \ref{Vir-on-V-L} that $L(-1)=\D$ and $L(0)$ is the grading operator,
 thus $V_{B}$ is a semi-conformal vertex superalgebra.
\end{proof}

\begin{de}\label{invarant-bilinear-form}\cite{fhl}
Let $V$ be a vertex operator superalgebra and let $M=M_{\bar{0}}\oplus M_{\bar{1}}$ be a $({\bf \Z}/2{\bf \Z})$-graded $V$-module, a bilinear form $(\cdot,\cdot)$ on $M$  is said to be \emph{invariant} if
\begin{gather*}
(M_{i},M_{j})=\{0\},\quad \mbox{if}\quad i\neq j ;\\
(Y(v,z)w,w')=\varepsilon_{v,w}(w,Y(e^{zL_{1}}(-z^{-2})^{L_{0}}v,z^{-1})w')
\end{gather*}
for $v\in V,\ w,w'\in M$.
\end{de}
Using the definition of semi-conformal vertex superalgebra and the Theorem 3.1 of \cite{Li94}, we have :
\begin{rem}
Let $V$ be a semi-conformal vertex superalgebra with $V_{(0)}=A$ and $V_{(1)}=B$. 
Then there exists a linear isomorphism from the space of invariant bilinear forms on $V$ to $\Hom_{\C}(A/L(1)B, \C)$.
\end{rem}

\section{Example of vertex superalgebroid}
Let $A$ be a commutative associative algebra with identity (over $\C$). 
Let $\g$ be a Lie superalgebra acting on $A$ as a Lie superalgebra of derivations. 
That is, we are given a Lie superalgebra homomorphism $\pi: \g\rightarrow {\rm Der}(A)$. In this section, we will use $ga:=\pi(g)a$, for any $g\in\g,a\in A$.

Since we assume $A$ to be purely even, the odd part of $\g$ necessarily acts trivially on $A$. One can consider also $A$ to be a super-commutative associative superalgebra but then the signs in the formulas become more complicated \cite{DHK}.

Now we introduce the definition of a Lie superalgebroid (Refs.\cite{gms} and \cite{DHK}).
\begin{de}
A Lie superalgebroid is a pair $(\g, A)$, where $A$ is a super-commutative associative superalgebra, $\g$ is a Lie superalgebra  equipped
with an $A$-module structure and a left-module action on $A$ by
derivation such that
\begin{eqnarray}
[u,av]&=&\varepsilon_{u, a}a[u,v]+(ua)v,\\
a(ub)&=&(au)b\;\;\;\mbox{ for }u,v\in \g,\; a,b\in A.
\end{eqnarray}
\end{de}

In \cite{DHK}, they  show that  $A\otimes \g$  has a Lie $A$-superalgebroid structure.
The Lie bracket on $A\otimes \g$ is given by
\begin{eqnarray}\label{liealgebroid}
[a\otimes g, a'\otimes g']=aa'\otimes[g, g']+a(ga')\otimes g'-\varepsilon_{g, g'}a'(g'a)\otimes g
\end{eqnarray}
for $a,a'\in A$, $g,g'\in \g$ (with $g,g'$ homogeneous),  and  the action of $A\otimes \g$ on $A$ is given by
\begin{eqnarray}
(a\otimes g)a'=a(ga'), \    \   \mbox{ for }a,a'\in A,\  g\in \g.
\end{eqnarray}

Set $T=A\otimes \g$ and let  $\Omega=A\partial A$ be the $A$-module of K\"ahlerian 1-differentials with 
 the canonical $A$-derivation $\partial: A\longrightarrow\Omega$, i.e., 
 $\Omega$ is the quotient $A$-module of the free $A$-module
 over $\partial A$ modulo relations for $a, a'\in A$: 
\begin{eqnarray}\label{seven}
\partial(aa')=a\partial a'+a'\partial a. 
\end{eqnarray}
We have the canonical bilinear pairing
$\<\cdot,\cdot\>:  {\rm Der}(A)\times \Omega\rightarrow A$, $\<\tau, a\partial b\>=a\tau(b),$ for $a,b\in A,\tau\in {\rm Der}(A).$

Set $B=T\oplus A\partial A.$
The following result is clear.
\begin{lem} 
The space $B=A\otimes \g+A\partial A$ is an non-associative unital $A$-module with the action defined by 
\begin{eqnarray}
&&a*(b\otimes g)=ab\otimes g+(ga)\partial b+(gb)\partial a, \\
 &&a*(a'\partial b)=(aa')\partial b
 \end{eqnarray}
for $a,a',b\in A,\  g\in \g.$ 
\end{lem}

Extend $\pi: \g\rightarrow {\rm Der}(A)$ to $\pi: B\rightarrow {\rm Der}(A)$ by
$\pi(a\partial b)=0, \pi(a\otimes g)=a\pi(g).$
From this condition, we have $\pi\circ \partial =0$ and 
\begin{eqnarray}\label{four}
\pi(a*(b\otimes g))=a\pi(b\otimes g).
\end{eqnarray}

\begin{lem}
For any $a,a'\in A,\ v\in B$, we have 
\begin{eqnarray}\label{one}
a*(a'*v)-(aa')*v=\pi(v)(a)* \partial (a')+\pi(v)(a')*\partial (a).
\end{eqnarray}
\end{lem}

\begin{proof} Let $a,a'\in A,\  u\in A\partial A$. We have $a*(a'*u)-(aa')*u=0$, $\pi(u)=0$.
For any $a,a'\in A,\ b\otimes g\in T$, 
\begin{eqnarray*}
a*(a'*(b\otimes g))&=&a*(a'b\otimes g+(ga')\partial b+(gb)\partial a')\\
&=&aa'b\otimes g+(ga)\partial (a'b)+g(a'b)\partial a\\
&&+a(ga')\partial b+a(gb)\partial a'\\
&=&aa'b\otimes g+(ga)(a'\partial b+b\partial a')+(ga')b\partial a\\
&&+a'(gb)\partial a+a(ga')\partial b+a(gb)\partial a',\\
\ (a a')*(b\otimes g)&=&aa'b\otimes g+(g(aa'))\partial b+(gb)\partial (aa')\\
&=&aa'b\otimes g+((ga)a'+a(ga'))\partial b\\
&&+(gb)(a\partial a'+a'\partial a),\\
\ a*(a'*(b\otimes g))-(a a')*(b\otimes g)&=&b(ga)\partial a'+b(ga')\partial a\\
&=&\pi(b\otimes g)(a)* \partial (a')+\pi(b\otimes g)(a')*\partial (a).
\end{eqnarray*}
\end{proof}

\begin{lem}
Define a linear operation (bracket) $[\cdot,\cdot]$ on $B= (A\otimes \g )\oplus A\partial A$ by 
$$[A\partial A,A\partial A]=0,$$
 $$[a'\otimes g', a\partial b]=a'(g'a)\partial b+a\partial(a'(g'b)),$$ 
$$[a\partial b, a'\otimes g']=-a'(g'a)\partial b+a'(g'b)\partial a,$$
$$[a\otimes g, a'\otimes g']=aa'\otimes[g, g']+a(ga')\otimes g'-\varepsilon_{g, g'}a'(g'a)\otimes g$$
for $a,a',b,b'\in A,\    g,g'\in \g.$ 
Then  $B$ is a Leibniz superalgebra.
\end{lem}

\begin{proof}  
Recall that $A\otimes \g$  is a Lie $A$-superalgebroid. Thus the case that all three elements belong to $A\otimes \g$ satisfies the graded  Leibniz  identity.

For any $a,a',b,c\in A,\  g,g'\in \g$,
\begin{eqnarray*}
[c\otimes g,[a'\otimes g',a\partial b]]&=&c(ga')(g'a)\partial b+ca'(g(g'a))\partial b+a'(g'a)\partial(c(gb))\\
&&+c(ga)\partial(a'(g'b))+a\partial(c(ga')(g'b))+a\partial(ca'g(g'b))\\
&=&[[c\otimes g,a'\otimes g'],a\partial b]+\varepsilon_{g, g'}[a'\otimes g',[c\otimes g,a\partial b]].
\end{eqnarray*}
For other cases, one can verify similarly. 
 \end{proof}

By definition, we have
\begin{eqnarray}\label{eight}
[a\otimes g,\partial (b)]=\partial (\pi(a\otimes g)(b))
\end{eqnarray}
and 
\begin{eqnarray}
~[a\otimes g, b*(a'\partial b')]&=&\pi(a\otimes g)(b)*(a'\partial b')+b*[a\otimes g, a'\partial b'],\label{two1}\\
~[a'\partial b',b*(a\otimes g)]&=&b*[a'\partial b',a\otimes g].\label{two2}
\end{eqnarray}
 
\begin{lem}
 Extend $\<\cdot,\cdot\>: Der(A)\times \Omega\rightarrow A$ to a bilinear pairing
$\<\cdot,\cdot\>: B\otimes_{\C}B\rightarrow A$, where 
$$\<a\partial b, a'\partial b'\>=0, \,\, \<a\otimes g,a'\otimes g'\>=0,$$
\begin{eqnarray}\label{last}
\<a\otimes g, a'\partial b\>=a\<g,a'\partial b\>=aa'\pi(g)b=\pi(aa'\otimes g)b.
\end{eqnarray}
Then for any $a\in A,\   u,v\in B$, we have 
\begin{eqnarray}\label{five}
\<a*u,v\>=a\<u,v\>-\pi(u)(\pi(v)(a)).
\end{eqnarray}
\end{lem}

\begin{proof} 
Indeed, for any $a\in A, u, v\in A\partial A$, $\<a*u,v\>=0, a\<u,v\>=0$, $\pi(u)=0, \pi(v)=0,$ it is easy to see  (\ref{five}).

For any $a,a_1,a',b\in A,g\in \g$, we have $\<a*(a_1\otimes g),a'\partial b\>=\<aa_1\otimes g,a'\partial b\>=aa_1a'(gb)$, $\<a_1\otimes g,a'\partial b\>=a_1a'(gb)$, and $\pi(a_1\otimes g)=0.$

For any $a,b,c\in A,g,g'\in \g$, 
\begin{eqnarray*}
\pi(b\otimes g)(\pi(c\otimes g')(a))&=&b(gc)(g'a)+bcg(g'a)\\
a\<b\otimes g,c\otimes g'\>&=&-abg'(gc)-acg(g'b)-a(gc)(g'b)\\
\<a*(b\otimes g),c\otimes g'\>&=&-abg'(gc)-cg((g'a)b+a(g'b))-(gc)(g'a)b\\
&&-a(gc)(g'b)+c(ga)(g'b)+c(gb)(g'a)\\
&=&-abg'(gc)-bcg(g'a)-acg(g'b)-b(gc)(g'a)-a(gc)(g'b)
\end{eqnarray*}
We obtain that $\<a*(b\otimes g),c\otimes g'\>=a\<b\otimes g,c\otimes g'\>-\pi(b\otimes g)(\pi(c\otimes g')(a)).$
\end{proof}

\begin{lem} 
We have 
\begin{eqnarray}\label{six}
\pi(v)(\<v_{1},v_{2}\>)&=&\<[v,v_{1}],v_{2}\>+\varepsilon_{v, v_{1}}\<v_{1},[v,v_{2}]\>
\end{eqnarray}
for any $v,v_1,v_2\in B$, and 
\begin{eqnarray}\label{three}
[u,v]+\varepsilon_{u,v}[v,u]=\partial(\<u,v\>)
\end{eqnarray}
for any $u,v\in B$.
\end{lem}
 
\begin{proof}
In fact, if all the three elements $v,v_1,v_2\in A\partial A$, it is easy to see  (\ref{six}). If $v\in A\partial A$, one of the elements $v_1,v_2$ $\in A\partial A$, it is also easy to see  (\ref{six}).
If $v=b\partial c\in A\partial A$ and $v_1=a\otimes g\in T$, $v_2=a'\otimes g'\in T$, we have
$\pi(v)=0$,
\begin{eqnarray*}
\<[v,v_{1}],v_{2}\>+\varepsilon_{v, v_{1}}\<v_{1},[v,v_{2}]\>&=&\<[b\partial c,a\otimes g],a'\otimes g'\>+\<a\otimes g,[b\partial c,a'\otimes g']\>\\
&=&\<-a(gb)\partial c+a(gc)\partial b,a'\otimes g'\>\\
&&+\<a\otimes g,-a'(g'b)\partial c+a'(g'c)\partial b\>\\
&=&-aa'(gb)(g'c)+aa'(gc)(g'b)\\
&&+aa'(g'c)(gb)-aa'(g'b)(gc)\\
&=&0
\end{eqnarray*}
One can check the other cases by a similarly argument.

For the second part.
Indeed, for any $u,v\in A\partial A$ and $u\in A\partial A, v\in A\otimes\g$, the result is obviously.
For any $a,a'\in A,g,g'\in\g$, we have $[a\otimes g,a'\otimes g']+\varepsilon_{g, g'}[a'\otimes g',a\otimes g)]=0$, and $\partial(\<a\otimes g,a'\otimes g'\>)=0.$
\end{proof}
Recall $T$ is a Lie A-superalgebroid, we have 
\begin{eqnarray}\label{two3}
[a\otimes g,b*(a'\otimes g')]=\pi(a\otimes g)(b)*(a'\otimes g')+b*[a\otimes g,a'\otimes g'].
\end{eqnarray}
We can verify (\ref{pe2.22}) by $(\ref{one})$, $(\ref{pe2.23})$  by $(\ref{two1}), (\ref{two2}) ~and ~(\ref{two3})$, $(\ref{e2.23})$-$(\ref{e2.29})$ by $(\ref{three})$, $(\ref{four})$, $(\ref{five})$, $(\ref{six})$, $(\ref{seven})$, $(\ref{eight})$, $(\ref{last})$ respectively. Thus
we get $B$ is a vertex $A$-superalgebroid.

Recall Proposition \ref{pact}, for any $a, a'\in A,\ g\in \g$, we have
\begin{eqnarray*}
a_0A=0,\   a_0 (A\partial A)=0=(A\partial A)_0a,\   a_0(a'\otimes g)=-(a'\otimes g)_0a=-a'(ga).
\end{eqnarray*}
\begin{de}
A derivation from a Lie superalgebra $\g$ into a $\g$-module $A$ is an even linear map $d: \g\rightarrow A$ such that 
\begin{eqnarray*}
d([x,y])=x\circ d(y)-y\circ d(x), \quad x,y\in\g.
\end{eqnarray*}
\end{de}

\begin{lem}\label{EX}
Let $B=(A\otimes \g)\oplus A\partial A$ be the vertex $A$-superalgebroid associated to $(A,\g)$ and let $d: \g\rightarrow A$ be a derivation from  $\g$ into the $\g$-module $A$. Then $A\oplus B$ becomes a $\mathfrak{b}$-module with 
\begin{eqnarray*}
&&L(0)|_A=0, \   \  \  L(0)|_B=1, \    \    \   L(1)|_{A}=0,\   \   \  L(1)|_{A\partial A}=0,\\
&&L(1)(a\otimes g)=ga+ad(g)\   \   \   \mbox{ for }a\in A,\    g\in \g.
\end{eqnarray*}
 Furthermore this $\mathfrak{b}$-module satisfies (\ref{de}) and (\ref{xyzw}). \end{lem}
\begin{proof}
It is straightforward to show that $A\oplus B$ is a $\mathfrak{b}$-module.
For any $a, a'\in A, \   g\in \g$, we have
\begin{eqnarray*}
 L(1)(a(a'\otimes g))&=&L(1)(aa'\otimes g+ag(a')+a'(ga))\\
&=&g(aa')+aa'd(g)\\
&=&(ga)a'+a(ga')+aa'd(g),\\
\ aL(1)(a'\otimes g)-a_0(a'\otimes g)&=&a(ga')+aa'd(g)-a_0(a'\otimes g)\\
&=&a(ga')+aa'd(g)+(ga)a'.
\end{eqnarray*}
This proves that (\ref{xyzw}) holds.

Next we need to show that: for any $u,v\in B$, $L(1)(u_0v)=(L(1)u)_0v+u_0(L(1)v)$.
Recall Proposition \ref{pact}, we have $u_{0}v=[u,v]$, for any $u,v\in B$.

For $u,v\in A\partial A$, we have $L(1)u=0$, $L(1)v=0$ and $L(1)(u_0v)=0$.
For $u\in A\partial A$, $v\in A\otimes \g$, $L(1)u=0$, $u_0(L(1)v)=0$, $L(1)(u_0v)=0$. For $u\in A\otimes \g$, $v\in A\partial A$, $L(1)v=0$, $(L(1)u)_0v=0$, $L(1)(u_0v)=0$.

For $u=a\otimes g$, $v=a'\otimes g'$, we have
\begin{eqnarray*}
(L(1)u)_0v&=&(ga+ad(g))_0v=-a'g'(ga)-a'(g'a)d(g)-a'ag'(d(g)),\\
u_0(L(1)v)&=&u_0(g'a'+a'd(g'))=ag(g'a)+a(ga')d(g')+a'ag(d(g')),\\
u_0v&=&[a\otimes g, a'\otimes g']=aa'\otimes[g, g']+a(ga')\otimes g'-\varepsilon_{g, g'}a'(g'a)\otimes g.
\end{eqnarray*}
Then 
$L(1)(u_0v)=ag(g'a)-a'g'(ga)+aa'd([g,g'])-\varepsilon_{g, g'}a'(g'a)d(g)+a(ga')d(g')$.
Note that $d$ is a derivation, i.e. $d([g,g'])=gd(g')-g'd(g)$, thus we have $L(1)(u_0v)=(L(1)u)_0v+u_0(L(1)v).$

This proves that (\ref{de}) holds.
\end{proof}
In summary, we have proved:
 
  \begin{prop}
Let $A$ be a commutative associative algebra and let $\g$ be a Lie superalgebra which acts on $A$ as a Lie algebra of derivations. 
For any derivation $d: \g \rightarrow A$,  on
the associated vertex (super)algebra $V_{B}$ with $B=(A\otimes \g) \oplus A\partial A $
there exists a semi-conformal structure, which is uniquely determined by
\begin{eqnarray}
&&L(0)|_A=0, \   \  \  L(0)|_B=1, \    \    \   L(1)|_{A}=0,\   \   \  L(1)|_{A\partial A}=0,\\
&&L(1)(a\otimes g)=ga+ad(g)\   \   \   \mbox{ for }a\in A,\    g\in \g.
\end{eqnarray}
Furthermore, we have $L(1)(V_B)_{(1)}=L(1)B$. 
\end{prop}

\section*{Acknowledgement}
This work was done during the author's visit at Rutgers University under the host of Professor Haisheng Li in 2018. We would like to thank Professor Haisheng Li for his  generous help with the preparation of this paper. The author also wants to thank  Professor Shaobin Tan for his valuable discussions.

\end{document}